\documentclass{amsart}

\usepackage{color}
\usepackage{amsmath}%
\usepackage{mathrsfs}
\usepackage{amsfonts}%
\usepackage{amssymb}%
\usepackage{graphicx}
\usepackage{bm}
\usepackage{verbatim}

\newtheorem{theorem}{Theorem}[section]
\newtheorem{corollary}[theorem]{Corollary}
\newtheorem{definition}[theorem]{Definition}

\newtheorem{lemma}[theorem]{Lemma}

\newtheorem{remark}[theorem]{Remark}
\numberwithin{equation}{section}

\newcommand{\nn}{\mathbb{N}}
\newcommand{\zz}{\mathbb{Z}}

\newcommand{\rr}{\mathbb{R}}

\newcommand{\set}[1]{\left\{ #1 \right\}}
\newcommand{\paren}[1]{\left( #1 \right)}

\newcommand{\sth}{\mid}

\newcommand{\norm}[1]{\left\lVert#1\right\rVert}

\newcommand{\ee}{\mathscr{E}}
\newcommand{\hh}{\mathscr{H}}

\DeclareMathOperator{\dom}{dom}

\renewcommand{\dim}{\operatorname{dim}}

\newcommand{\SG}{\text{SG}}
\newcommand{\notblock}{$(N-1)$-block}

\usepackage[toc,page]{appendix}

\begin{document}
\title{Harmonic Gradients on Higher Dimensional Sierpi\'nski Gaskets}
\author{Luke Brown}
\address[Luke Brown]{Drexel University\newline%
\indent Department of Mathematics}
\email{lcbrown@wpi.edu}
\author{Giovanni Ferrer}
\address[Giovanni Ferrer]{University of Puerto Rico (Mayag\"uez)\newline%
\indent Department of Mathematics}
\email{giovanni.ferrer@upr.edu}
\author{Gamal Mograby}
\address[Gamal Mograby]{University of Connecticut\newline%
\indent Department of Mathematics}
\email{gamal.mograby@uconn.edu}
\author{Luke G. Rogers}
\address[Luke Rogers]{University of Connecticut\newline%
\indent Department of Mathematics}
\email{luke.rogers@uconn.edu}
\author{Karuna Sangam}
\address[Karuna Sangam]{Rutgers University\newline%
\indent Department of Mathematics}
\email{karuna.sangam@rutgers.edu}

\thanks{Research supported in part by the NSF Grants DMS-1659643 and DMS-1613025.}
\date\today

\begin{abstract}
We consider criteria for the differentiability of functions with continuous Laplacian on the Sierpi\'nski Gasket and its higher-dimensional variants $SG_N$, $N>3$, proving results that generalize those of Teplyaev~\cite{sasha1}.  When $SG_N$ is equipped with the standard Dirichlet form and measure $\mu$ we show there is a full $\mu$-measure set on which continuity of the Laplacian implies existence of the gradient $\nabla u$, and that this set is not all of $SG_N$.  We also show there is a class of non-uniform measures on the usual Sierpi\'nski Gasket with the property that continuity of the Laplacian implies the gradient exists and is continuous everywhere, in sharp contrast to the case with the standard measure.
\end{abstract}

\maketitle
\pagestyle{headings}
\markleft{L. Brown, G. Ferrer, G. Mograby, L. G. Rogers, K. Sangam}

\section{Introduction}
In analysis on fractals the basic differential operator is a Laplacian obtained either by probabilistic methods~\cite{Barlowbook} or  as a renormalized limit of graph Laplacians~\cite{kigami1}.  There are then various approaches to defining a gradient, or first derivative, such as those in~\cite{Kusuoka2,Strichartz5,Kigami2,PelTep,Hino}, and related questions remain an active area of research~\cite{Hino2,Kajino,fabrice}.  The results of this paper are a contribution to understanding  what connection there is between smoothness measured using the Laplacian and the pointwise existence of a gradient, as the situation is very different than in the setting of Euclidean spaces or manifolds.

A fundamental result relating the regularity of the Laplacian and existence of a gradient was proven by Teplyaev in~\cite{sasha1}, who gave an example in which functions with continuous Laplacian are differentiable a.e.\ but can fail to be differentiable at a countable dense set of points.  The innovative idea was not the example itself, which was just the standard Sierpi\'nski gasket with its usual Laplacian and Bernoulli measure, but a concrete description of the gradient which allowed the points of differentiability to be described fairly precisely.  It should be noted that, on the Sierpi\'nski gasket, Teplyaev's gradient can be identified with that of Kusuoka~\cite{Kusuoka2}.

In Section~\ref{sec:gaskets} we introduce the $N$-vertex Sierpi\'nski Gasket $SG_N$ and its analytic structure, and in Section~\ref{sec:harmgradients} we review Teplyaev's gradient and some basic results from~\cite{sasha1}.  In Section~\ref{sec:dependenceonmu} we then build on Teplyaev's work to show how the structure of the measure affects the connection between Laplacian regularity and existence of a gradient.  Theorem~\ref{thm:uneqweightmsr} shows that, in contrast with the previously mentioned results from~\cite{sasha1}, if we equip the Sierpi\'nski gasket with a suitably chosen self-similar measure having unequal weights, then we find that functions with continuous Laplacian are not only differentiable everywhere but the gradient is continuous.  The discontinuity of natural gradients in the case of the standard measure is well-known, and it is rather unexpected that a continuous gradient can be obtained with such a simple modification of the measure.

Section~\ref{sec:diffonSGN} is concerned with differentiability results on the gaskets $SG_N$.  We show that if the self-similar Laplacian and  measure are symmetric under the symmetries of the underlying simplex then the results proved in~\cite{sasha1} can be generalized to $SG_N$, though the description of the points of differentiability is less explicit and some proofs are correspondingly more complicated.

\section{Higher Dimensional Sierpi\'nski Gaskets}\label{sec:gaskets}
Let $N \in \mathbb{N}$ with $N \geq 3$. We largely follow~\cite{Kigami2,dan1} in the following definitions and basic results. Note that, in the definition below, $SG_3$ is the usual Sierpi\'nski Gasket.

\begin{definition} \label{def:sgn}
Let $\{p_i \}^{N-1}_{i=0}$ be the vertices of a regular $N$ simplex in $\mathbb{R}^{N-1}$ such that $|p_j - p_k|=1$ if $j \neq k$. Let $\{F_i \}^{N-1}_{i=0}$, with $F_i: \mathbb{R}^{N-1} \rightarrow \mathbb{R}^{N-1}$, be the iterated function system defined by $F_j(x) = \frac{1}{2}(x-p_j) + p_j$. Then the $N$-dimensional Sierpi\'nski Gasket, denoted $SG_N$, is the unique non-empty compact set such that $SG_N = \bigcup^N_{j=0} F_j(SG_N)$. 
\end{definition}
\begin{definition} \label{def:words}
Let $S_N = \{0,1, \dots, N-1 \}$ and $\Omega_N = S_N^\mathbb{N}$ be  the collection of one-sided infinite words over $S_N$. Similarly, a finite word of length $m\in\mathbb{N}$ is an element of the $n$-fold product $S_N^m$.
\end{definition} 
For simplicity, we often omit the index $N$ in $\Omega_N$ and $S_N$ and write $\Omega$, $S$ respectively.
$SG_N$ is post-critically finite with post-critical set $V_0 = \{ p_0, \ldots , p_{N-1} \}$. 
We write $F_w = F_{w_1} \circ \ldots \circ F_{w_m}$, where $w=w_1 \ldots w_m$ is a finite word of length $m$ over the alphabet $S$.
Let $V_m = \bigcup_{w \in S^m} F_w(V_0)$ and  consider these points as vertices of a graph in which adjacency $x\sim_m y$ means there is a word $w$ of length $m$ such that $x,y \in F_{w}(V_0)$. 
A non-negative definite, symmetric, quadratic form on $SG_N$ may be defined 
as a limit of graph energies as follows.
\begin{definition} \label{def:energy}
Let $u,v$ be continuous functions on $SG_N$. The bilinear form
\[ \mathscr{E}_m(u,v) = \left(\frac{N+2}{N} \right)^m\sum_{x \underset{m}{\sim} y} \left(u(x)-u(y)\right)\left(v(x)-v(y)\right)\]
defines the graph energy of level $m$. 
\end{definition}
We write $\mathscr{E}_m(u)=\mathscr{E}_m(u,u)$. Then $\{\mathscr{E}_m(u) \}$ is a nondecreasing sequence of graph energies, so $\mathscr{E}(u) = \lim_{m\rightarrow \infty} \mathscr{E}_m(u)$ is well-defined; setting its domain to be $\mathop{dom}\mathscr{E} = \{ u:SG_N \to \mathbb{R} | \mathscr{E}(u)< \infty \}$ one obtains a non-negative definite, symmetric quadratic form that extends to the completion of $\cup_m V_m$, which can be shown to be $SG_N$, and the domain is uniform-norm dense in the continuous functions. For proofs of these facts see~\cite{kigami1}.  By construction this form is also self-similar in the sense that
\begin{equation}\label{eq:formselfsimonSGN}
\ee(f,f)=\Bigl( \frac{N+2}{N}\Bigr)\sum_{i=0}^{N-1}  \ee(f\circ F_i,f\circ F_i)
\end{equation}

      We equip $SG_N$ with a Bernoulli measure $\mu$ with
weights $\{0<\mu_i<1 \}^{N-1}_{i=0}$, $\sum_i\mu_i=1$, at which point $(\ee,\mathop{dom}\ee)$ is a Dirichlet form and we may define the Dirichlet $\mu$-Laplacian as follows (see~\cite{kigami1,bob1}):
\begin{definition} \label{def:laplacian}
Let $u \in \mathop{dom} \mathscr{E}$, and let $f$ be continuous. Then $u \in \mathop{dom}\Delta_\mu$ with $\Delta_\mu u = f$ if 
\begin{equation}
\mathscr{E}(u,v) = - \int_{\text{SG}_N} fv d\mu \text{ for all $v \in \mathop{dom_0}\mathscr{E}$}, \nonumber 
\end{equation}
where $\mathop{dom_0}\mathscr{E}$ is the subspace of $\mathop{dom}\mathscr{E}$ consisting of functions that vanish at $V_0$.
\end{definition}
A function $h \in \mathop{dom}\ee$ is called harmonic if it has specified values on $V_0$ and minimizes the graph energies
$\ee_n(u)$ for all $n \geq 1$. 
We can calculate $h|_{V_{m+1}}$ from $h|_{V_{m}}$ using  harmonic extension matrices.
\begin{definition}\label{def:harmextension}
Let $h$ be a harmonic function on $SG_N$. The harmonic extension matrices $\{A_i \}^{N-1}_{i=0}$ are defined by
\begin{eqnarray}
\begin{pmatrix}
h(F_i (p_0)) \\
\vdots \\
h(F_i (p_{N-1}))\\
\end{pmatrix}
= A_i
\begin{pmatrix}
h(p_0) \\
\vdots \\
h(p_{N-1})\\
\end{pmatrix}.
\nonumber
\end{eqnarray}
\end{definition}
The harmonic extension matrices for $SG_N$ are derived in~\cite{dan1}, and given by
\begin{eqnarray}
\label{eq:azero}
A_0 = \frac{1}{N+2}
 \begin{pmatrix}
  N+2 & \textbf{0}\\
  \textbf{2} & I_{N-1}+J_{N-1}\\
 \end{pmatrix}
\end{eqnarray}
where $I_{N-1}$ is the $(N-1)\times(N-1)$ identity matrix, $J_{N-1}$ is the $(N-1)\times(N-1)$ matrix with all entries equal $1$, and $\textbf{0}$ and $\textbf{2}$ are the $(N-1)$ size vectors with all entries 0 and 2 respectively. All other harmonic extension matrices can be found with cyclic row and column permutations.

Let $A_i$ be a harmonic extension matrix of $SG_N$ and $\sigma(A_i)$ the set of eigenvalues of $A_i$. Then the eigenspace of $\lambda \in \sigma(A_i)$, denoted by $E_i[\lambda]$, is the subspace of $\mathbb{R}^N$ spanned by the eigenvectors of $A_i$ corresponding to $\lambda$.
We need an elementary lemma.
\begin{lemma} \label{lem:eigenvalues2}
Let $A_i$ be a harmonic extension matrix for $SG_N$. Then, the eigenvalues of $A_i$ are
$\sigma(A_i)=\{1,\frac{N}{N+2},\frac{1}{N+2}\}$.
The corresponding eigenspaces have the following dimensions:
\begin{eqnarray}
dim \ E_i[1] & = & 1 \nonumber \\
dim \ E_i \big[ \tfrac{N}{N+2} \big] & = & 1 \nonumber \\
dim \ E_i \big[ \tfrac{1}{N+2} \big] & = & N-2 .\nonumber
\end{eqnarray}
\end{lemma}
\begin{proof}
It can be easily verified that
that $(1,\ldots,1)^T$ is a simple eigenvector with eigenvalue $1$, $(0,1,\ldots,1)^T$ is a simple eigenvector with eigenvalue
$\tfrac{N}{N+2}$ and that 
$(\ldots,0,1,-1,0,\ldots)^T$ are $N-2$ eigenvectors corresponding to the eigenvalue 
$\tfrac{1}{N+2}$.
\end{proof}
Let $\mathscr{H}$ denote the space of harmonic functions.  Since these are determined by their values on $V_0$ this space is $N$-dimensional. Let  $W\subset\mathscr{H}$ be the subspace of constant functions and $P:\mathscr{H} \rightarrow  \mathscr{H} / W=:\tilde{\mathscr{H}}$ be the quotient map.  
 As $W$ is the eigenspace for the eigenvalue $1$ we have $A_i (W) = \set{A_iw \sth w \in W} \subseteq W$ for $i \in \{0, \dots, N-1 \}$ and thus there is 
$\tilde{A}_i :\tilde{\mathscr{H}} \rightarrow \tilde{\mathscr{H}} $, such that 
$\tilde{A}_i \circ P = P \circ A_i$. We call $\{ \tilde{A}_i \}^{N-1}_{i=0}$ the induced harmonic extension matrices.
The energy $\mathscr{E}(\cdot,\cdot)$ is a bilinear form on $\mathscr{H}$ and $\mathscr{E}(u,u)=0$ if and only if $u$ is a constant function on $SG_N$, so the restriction of $\mathscr{E}(\cdot,\cdot)$ on $\tilde{\mathscr{H}}$ is a well-defined inner product that makes $\tilde{\mathscr{H}}$ a Hilbert space.
The following is an immediate consequence of Lemma \ref{lem:eigenvalues2}.
\begin{corollary} \label{lem:eigenvalues3}
Let $\tilde{A}_i$ be a induced harmonic extension matrix for $SG_N$. Then, the eigenvalues of $\tilde{A}^{-1}_i$ are
$\sigma(\widetilde{A}^{-1}_i)=\{\frac{N+2}{N},  N+2\}$.
The corresponding eigenspaces have the following dimensions:
\begin{eqnarray}
\dim \ E_i \big[ \tfrac{N+2}{N} \big] & = & 1 \nonumber \\
\dim \ E_i \big[ N+2 \big] & = & N-2 .\nonumber
\end{eqnarray}
\end{corollary}
\begin{remark}
\label{remark1}
A simple calculation shows that the eigenspaces $E_i \big[ \tfrac{N+2}{N} \big]$ and $E_i \big[ N+2 \big]$ of $\tilde{A}^{-1}_i$ are orthogonal subspaces in $(\tilde{\mathscr{H}},\mathscr{E})$.
\end{remark}

\section{Harmonic Gradients in the sense of Teplyaev}\label{sec:harmgradients}
We define a harmonic gradient on $\SG_N$ following the approach and notation of Teplyaev~\cite{sasha1}, which is closely related to work of Kusuoka~\cite{Kusuoka2}. We require some notation for a cell containing a point described by an infinite word and for a harmonic approximation to the function on such a cell.
\begin{definition} \label{def:truncword}
For $\omega = \omega_1 \omega_2 \cdots \in \Omega_N$ the truncated word $[\omega]_n\in S_N^n$ is $[\omega]_n=\omega_1 \omega_2 \cdots \omega_n$.
\end{definition}

\begin{definition}
\label{def:gradient}
The $n$-level harmonic approximation at word $\omega \in \Omega$ is 
$$\nabla_n f(\omega) = \tilde{A}^{-1}_{[\omega]_n}\tilde{H}(f \circ F_{[\omega]_n}),$$
where $\tilde{H}(g) = P H(g)$, and $H(g)$ is the unique harmonic function that coincides with $g$ on the boundary of $SG_N$. The harmonic gradient at $\omega$ is defined to be 
$$\nabla f(\omega) = \lim_{n \rightarrow \infty} \nabla_n f(\omega)$$
if the limits exist in $\tilde{\mathscr{H}}$.
\end{definition}

Observe that the preceeding is analogous to the way in which secants converge to a tangent in elementary calculus, with harmonic functions playing the role of linear functions (because the latter are harmonic on $\mathbb{R}$). The $n$-level harmonic gradient of $f$ at a point $x$ is akin to a secant modulo constant functions because it is the unique globally harmonic function modulo constants that agrees with $f$ at the boundary points of a cell containing the point.  The matrices $ \tilde{A}^{-1}_{[\omega]_n}$ are used simply to find the boundary values of this harmonic function from data on the cell at scale $n$.  Then the limit of the harmonic approximations as the scale goes to zero is the harmonic gradient.  

The following theorems are essential to our treatment of the topic, and were proved for a resistance form satisfying the identity
\begin{equation}\label{eq:formisselfsim}
\ee(f,f)= \sum_{i=0}^{N-1}  r_i^{-1} \ee(f\circ F_i,f\circ F_i) 
\end{equation}
and a Bernoulli measure with weights $0<\mu_i<1$ in~\cite{sasha1}.

\begin{theorem}[\protect{\cite[Theorem~1]{sasha1}}]\label{thm:sasha1thm1}
Suppose $f\in\dom\Delta_{\mu}$. Then, $\nabla f(\omega)$ exists for every $\omega\in\Omega$ such that
\[ \sum_{n\geq1}r_{[\omega]_n}\mu_{[\omega]_n}||\widetilde{A}_{[\omega]_n}^{-1}|| < \infty \]
\end{theorem}

\begin{corollary}[\protect{\cite[Corollary~5.1]{sasha1}}]\label{cor:sasha1cor5.1}
Suppose that $f\in \dom\Delta_{\mu}$. Then, $\nabla f(\omega)$ exists for all $\omega\in\Omega$ if
$$ r_j\mu_j ||\widetilde{A}_{j}^{-1} ||<1 $$
For $j=1,\ldots,N$. Moreover, in this case, $\nabla f(\omega)$ is continuous in $\Omega$.
\end{corollary}

Teplyaev~\cite{sasha1} points out that Corollary~\ref{cor:sasha1cor5.1} is not applicable to  the Sierpi\'nski Gasket when $\mu$ is the standard (uniform) Bernoulli measure.  The same is true for $\SG_N$ for any $N\geq3$ because one may readily compute that $\|A_j^{-1}\|=N+2$ for each $j=0,\dotsc,N-1$, while each $\mu_j=N^{-1}$ and, as previously noted (see~\eqref{eq:formselfsimonSGN}), each $r_j=\frac{N}{N+2}$, so that $r_j\mu_j\|\tilde{A}_j^{-1}\|=1$.  Teplyaev shows that one can apply Theorem~\ref{thm:sasha1thm1} to certain points on the Sierpi\'nski Gasket, but their description is rather complicated.

\section{A measure on $\SG_3$ for which functions with continuous Laplacian have continuous gradient}\label{sec:dependenceonmu}

On the standard Sierpi\'nski Gasket $\SG_3$ with its usual self-similar resistance form (as defined in Section~\ref{sec:gaskets}) we consider the Laplacian associated to a non-uniform Bernoulli measure defined using the iterated function system of the second level, meaning that the similarities are compositions $F_{ij}=F_i\circ F_j$ of the usual three contractions on $\SG_3$.  The following theorem gives a condition on the measure sufficient to ensure functions with continuous Laplacian have continuous gradients.

\begin{figure}
\label{fig:crazyreu}
\centering
	\includegraphics[width=.4\textwidth,height=.346\textwidth]{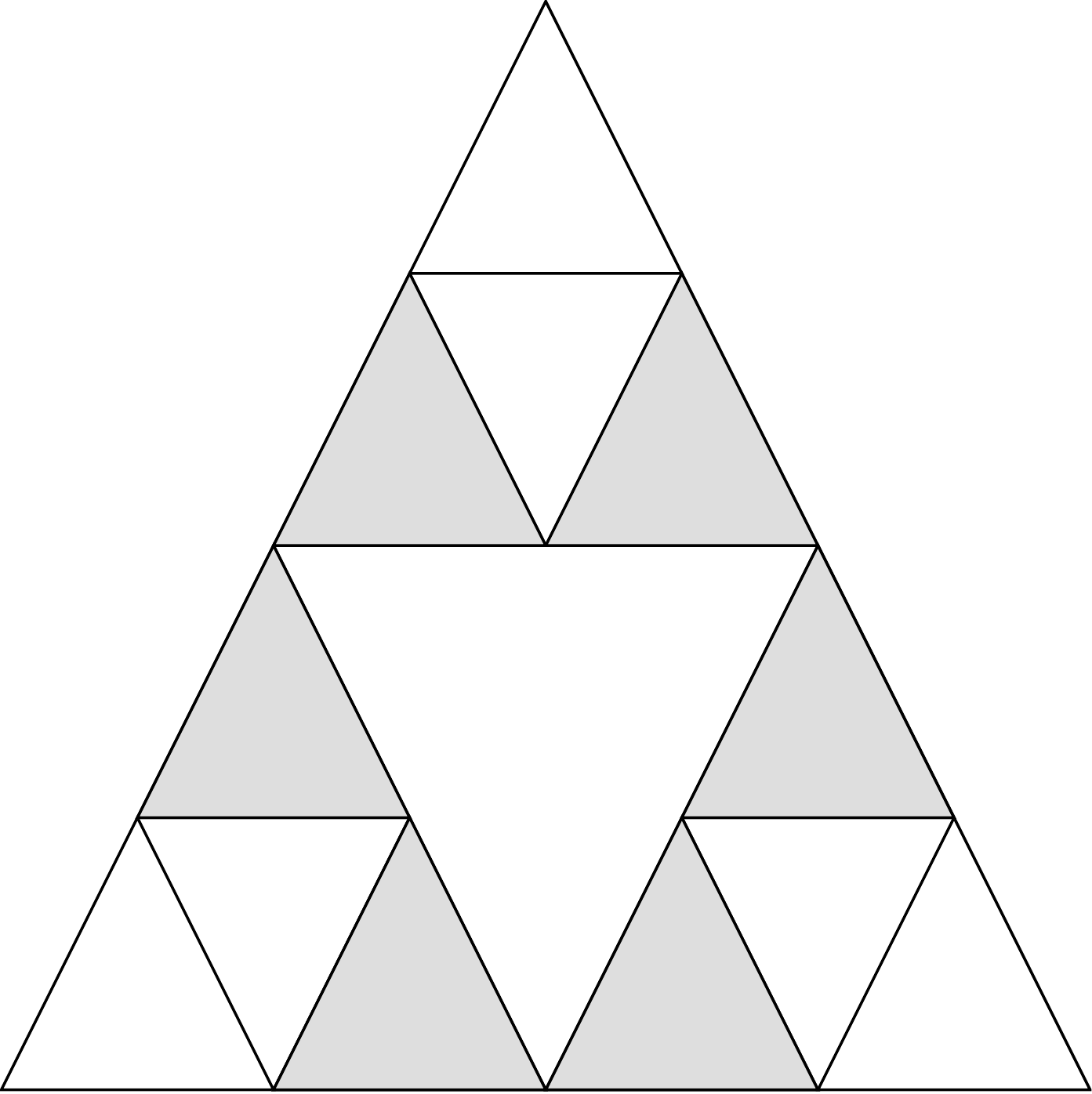}
\ \ 
	\includegraphics[width=.4\textwidth,height=.346\textwidth]{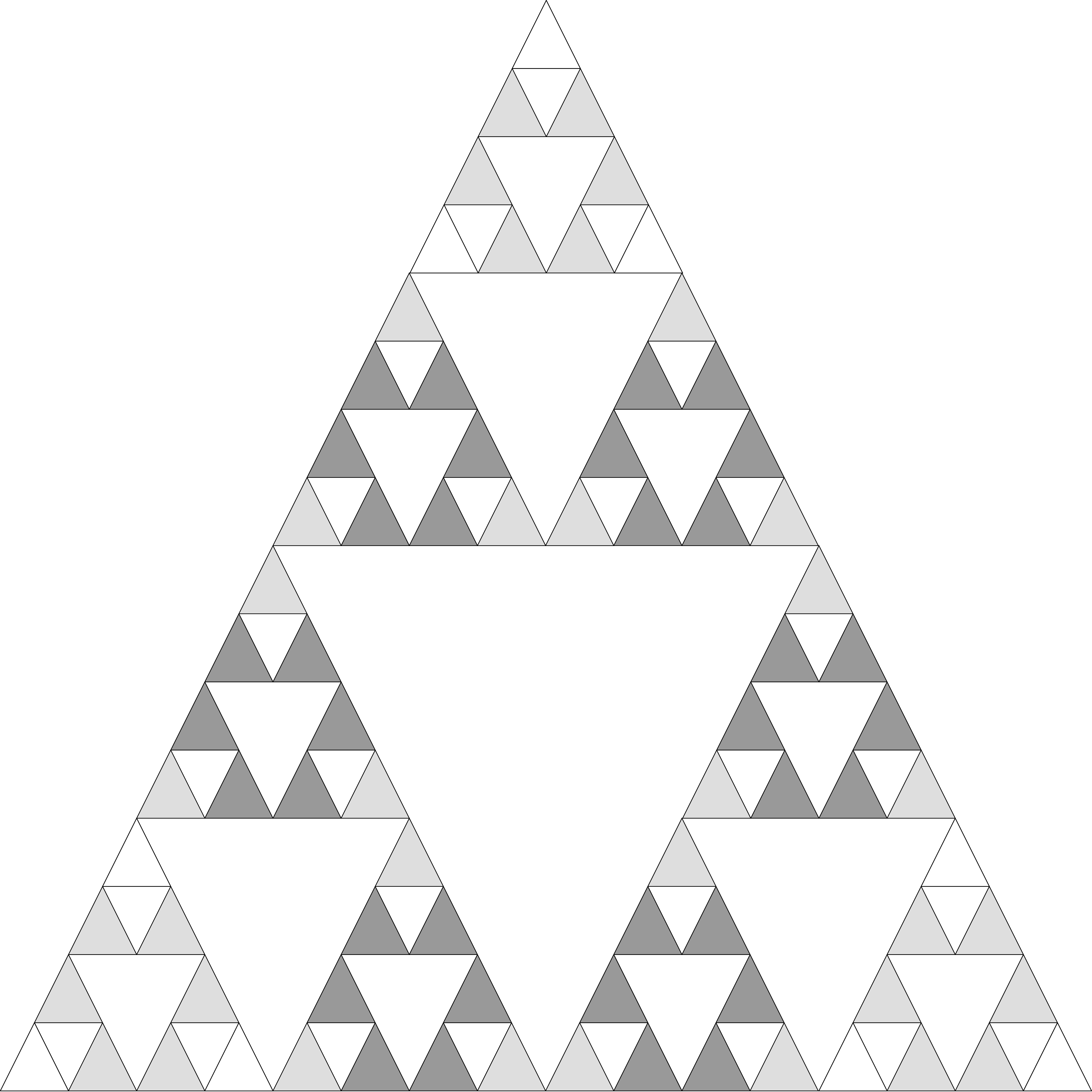}
\caption{$SG_3$ Level $1$ (left) and $2$ (right) with respect to the IFS $\{F_{ij} \}_{i,j \in S}$. Cells shaded according to their measure weights.}
\end{figure}

\begin{theorem}\label{thm:uneqweightmsr}
Let $\mu$ be the Bernoulli measure on $SG_3$ with the weights $\{\mu_{ij}\}_{i,j \in S}$ corresponding to the iterated function system $\{F_{ij}=F_i\circ F_j\}_{i,j \in S}$, so
\begin{eqnarray}
\mu(A)=\sum_{i,j \in S} \mu_{ij} \  \mu (F_{ij}^{-1}(A)) \nonumber
\end{eqnarray}
for any Borel set $A \subset SG_3$. If $\{\mu_{ij}\}_{i,j \in S}$ satisfy
\begin{equation}\label{eq:muvaluesforSGexample}
 \mu_{ij} < \begin{cases}
\frac{1}{9} & i = j \\
\frac{1}{\sqrt{17+4\sqrt{13}}} & i \neq j
\end{cases},
\end{equation}
then $u\in\operatorname{dom}\Delta_{\mu}$ implies that $\nabla u(\omega)$ exists and is continuous for all $\omega\in\Omega$.
\end{theorem}
\begin{remark}
The theorem provides examples because $\sqrt{17+4\sqrt{13}}<9$, so there are many choices of $\mu_{ij}$ satisfying both~\eqref{eq:muvaluesforSGexample} and $\sum_{i,j\in S}\mu_{ij}=1$.
\end{remark}

\begin{proof}
We apply Corollary~\ref{cor:sasha1cor5.1}, for which purpose we need the values of $r_{ij}$, $\mu_{ij}$ and an estimate of the norms of the matrices $\tilde{A}_{ij}^{-1}$ where $\tilde{A}_{ij}$ is the reduced harmonic extension matrix of the composition $F_{ij}=F_i\circ F_j$.  The $\mu_{ij}$ values are given in~\eqref{eq:muvaluesforSGexample} and $r_{ij} = (\frac{3}{5})^2$ because  the energy scaling for any $F_i$ is $\frac35$, as is apparent by comparing equations~\eqref{eq:formselfsimonSGN} and~\eqref{eq:formisselfsim}.

We can calculate the spectral radius $\rho((\tilde{A}^{-1}_{ij})^\ast \tilde{A}^{-1}_{ij})$  for each $i$ and $j$, using the description in Corollary~\ref{lem:eigenvalues3} and Mathematica, to obtain
\begin{eqnarray}
\sqrt{\rho[(\tilde{A}_{ij}^{-1})^\ast \tilde{A}_{ij}^{-1} ]} = \sqrt{\rho[(\tilde{A}_i^{-1}\tilde{A}_j^{-1})^\ast (\tilde{A}_i
^{-1} \tilde{A}_j^{-1}) ]} = \begin{cases}
25 & i = j \\ \frac{25}{9}\sqrt{17+4\sqrt{13}} & i \neq j
\end{cases}. \nonumber
\end{eqnarray}
Thus, we see that
\begin{eqnarray}
 r_{ij} \mu_{ij} \sqrt{\rho[(\tilde{A}_{ij}^{-1})^\ast \tilde{A}_{ij}^{-1} ]} = \begin{cases}
9 \mu_{ij} & i = j \\
\mu_{ij} \sqrt{17+4\sqrt{13}} & i \neq j
\end{cases}, \nonumber
\end{eqnarray}
and hence from~\eqref{eq:muvaluesforSGexample} that Corollary~\ref{cor:sasha1cor5.1} is applicable because the spectral radius dominates the norm.  It follows that if $u\in\operatorname{dom}\Delta_{\mu}$ then $\nabla u(\omega)$ exists and is continuous for all $\omega\in\Omega$. In Figure~\ref{fig:crazyreu}, we illustrate one such $\mu$.
\end{proof}

A similar argument works on $\SG_N$ for any $N\geq3$, though we do not know a convenient procedure for determining the optimal weights (corresponding to those in~\eqref{eq:muvaluesforSGexample}) if $N>3$.

\section{Gradients on $SG_N$ with the Standard Bernoulli Measure}\label{sec:diffonSGN}

As we noted at the end of Section~\ref{sec:harmgradients}, Corollary~\ref{cor:sasha1cor5.1} is not applicable to the Sierpi\'nski gasket or any $SG_N$, $N\geq3$ when they are equipped with the standard (fully symmetric) measure and Dirichlet form.  Hence in this setting there may be functions $u$ with continuous Laplacian but for which $\nabla u$ fails to exist, at least at some points. Indeed, in~\cite{sasha1}, Teplyaev gives an example which may be used to construct a function $u$ with continuous Laplacian such that $\nabla u$ is undefined on a countable dense set.  This example may readily be generalized to $SG_N$, $N>3$.  However, Teplyaev also proves there is a full $\mu$-measure set of words $\omega\in\Omega$ for which continuity of $\Delta u$ implies existence of $\nabla u(\omega)$. The purpose of this section is to prove a generalization of this result to $SG_N$, $N>3$.

The key idea in Teplyaev's proof of the result mentioned above is that harmonic functions have an improved scaling behavior near points defined by words that are asymptotically sufficiently non-constant.  An appropriate generalization to our context uses the following concept.
\begin{definition} \label{def:block}
A \textbf{$\bm{k}$-block} for an alphabet $S_N = \set{0,1,...,N-1}$ is a length $k$ word $w$ with $k$ distinct letters, meaning $w=w_1 w_2 \cdots w_k$ such that each $w_i \in S_N$ and $w_i \neq w_j$ for all $i \neq j$.
\end{definition}

The key scaling behavior for an \notblock\ is the following estimate.
\begin{lemma}\label{lem:betaN}
Fix $N$.  There is $\beta_N<1$ such that for any \notblock\ $w\in S_N^{N-1}$
\begin{equation*}
	\frac1{N+2}\Bigl\| \tilde{A}^{-1}_w \Bigr\|^{1/(N-1)} \leq \beta_N.
	\end{equation*}
\end{lemma}
\begin{proof}
Since the set of \notblock s is finite it suffices to show the estimate for an arbitrary \notblock\ $w=w_1\dotsc w_{N-1}$.  Then $\tilde{A}^{-1}_w= \tilde{A}^{-1}_{w_{N-1}}\dotsm \tilde{A}^{-1}_{w_1}$ and the maximal eigenvalue for each $ \tilde{A}^{-1}_{w_j}$ is $N+2$, so the result is true unless there is a  vector common to the $(N+2)$-eigenspaces of all of the  $\tilde{A}^{-1}_{w_j}$.  However we determined these eigenspaces explicitly in the proof of Lemma~\ref{lem:eigenvalues2}.  Recalling that passage from $A_i$ to $\tilde{A}_i$ eliminated the constant eigenspace (which was common to all $A_i$), we see that the eigenvectors of $\tilde{A}_i^{-1}$ with eigenvalue $N+2$ correspond to vectors in $\mathbb{R}^N$ that are orthogonal to the constants and to the unit vector in  the $i^{\text{th}}$ direction.  A vector common to the eigenspaces of all $\tilde{A}^{-1}_{w_j}$ would then need to be orthogonal to the constants and to the unit vector in the $w_j$ direction for $j=1,\dotsc,N-1$, thus to all of $\mathbb{R}^N$.  This shows the estimate for an arbitrary \notblock\ and proves the lemma.
\end{proof}

\begin{remark}
One can compute $\beta_N$ explicitly, but we do not know an elementary way to do this for general $N$.  In~\cite{sasha1} it is shown that $\beta_3=\sqrt{\frac{7+\sqrt{13}}{18}}$.
\end{remark}

The significance of a \notblock\ from our perspective is that the harmonic gradient exists at the point $F_w(X)$  if $w$ has sufficient asymptotic density of \notblock s.  The density is counted using the following. 

\begin{definition} \label{def:countingfunction}
The \textbf{block counting function} $C_N \colon \Omega \times \nn \to \nn \cup \set{0}$ is defined for  $\omega \in \Omega$ by
\[C_N(\omega,n) = \# \{ i \in \nn \sth i \leq n - (N-2), \ [i, i+N-2] \text{ is an $(N-1)$-block}\}.\]
\end{definition}
The following theorem now provides a criterion sufficient for existence of the harmonic gradient.

\begin{theorem} \label{thm2}
Let $u\colon SG_N \to \rr$ and suppose $\Delta_\mu u$ is continuous, where $\mu$ is the standard Bernoulli measure.
Then $\nabla u(\omega)$ is defined at every $\omega \in \Omega$ such that
\begin{equation}
 \liminf_{n\to\infty}\frac{C_N(\omega,n)}{\log n} >\frac1{|\log \beta_N|}.\label{thm2eqn}
\end{equation}
\end{theorem}

\begin{lemma}
\label{lemma4}
Let $\omega \in \Omega_N$. Then
\begin{equation*}
\bigl\|\tilde{A}^{-1}_{[\omega]_n}\bigr\|\leq (N+2)^n \beta_N^{C_N(\omega,n)}.
\end{equation*}
\end{lemma}
\begin{proof}
The proof is inductive with base case $n=1$, for which $C_N(\omega,1)=0$ (by definition) and we are bounding the norm of $\tilde{A}^{-1}_{w_1}$ by its maximal eigenvalue.  For the inductive step we consider two cases.

If the last $N-1$ letters  of $[\omega]_{n+1}$ do not form a \notblock\ then $C_N(\omega,n) = C_N(\omega,n+1)$ and bounding the norm of $\tilde{A}^{-1}_{w_{n+1}}$ by the maximal eigenvalue $(N+2)$ we have from the induction hypothesis
\[
\norm{\tilde{A}^{-1}_{[\omega]_{n+1}}}
\leq (N+2) \norm{\tilde{A}^{-1}_{[\omega]_n}}
\leq (N+2)^{n+1} \beta_N^{C_N(\omega,n)}
=(N+2)^{n+1} \beta_N^{C_N(\omega,n+1)}.
\]
In the other case, where the last $N-1$ letters form a \notblock, we instead use Lemma~\ref{lem:betaN} on this block and the inductive bound on for $n+1-(N-1)$ to obtain
\begin{align*}
\norm{\tilde{A}^{-1}_{[\omega]_{n+1}}}
& \leq \norm{\tilde{A}^{-1}_{[\omega]_{n-N+2}}} \norm{\tilde{A}^{-1}_{\omega_{n-N+3} } \dots \tilde{A}^{-1}_{\omega_{n+1} }} \\
& \leq (N+2)^{n-N+2}\beta_N^{C_N(\omega,n-N+2)} (N+2)^{N-1} \beta_N^{N-1} \\
& \leq (N+2)^{n+1}\beta_N^{C_N(\omega,n-N+2)+N-1} \\
&\leq (N+2)^{n+1} \beta_N^{C_N(\omega,n+1)},
\end{align*}
where we also used the fact that $C_N(\omega,n+1) -C_N(\omega,n-N+2)\leq N-1$, which is immediate from the definition.
\end{proof}

\begin{proof}[Proof of Theorem~\ref{thm2}]
For the standard Bernoulli measure and Dirichlet form on the higher dimensional Sierpi\'nski Gasket we have $r_{[\omega]_{n}} \mu_{[\omega]_{n}}=\frac{1}{(N+2)^{n}}$, as noted after Corollary~\ref{cor:sasha1cor5.1}.  Inserting this and the result of Lemma~\ref{lemma4} we compute, using that $0<\beta_N<1$ that
\begin{align*}
	\sum_{n=1}^\infty r_{[w]_n} \mu_{[w]_n} \bigl\| \tilde{A}^{-1}_{[w]_n}\bigr\| 
	\leq	\sum_{n=1}^\infty \beta_N^{C_N(\omega,n)}
	\leq \sum_{n=1}^\infty n^{ -C_N(\omega,n)|\log \beta_N|/\log n}
\end{align*}
which is convergent because 
\begin{equation*}
	\frac{C_N(\omega,n)\log|\beta_N|}{\log n}>\frac12\biggl(1+\liminf_n \frac{C_N(\omega,n)\log|\beta_N|}{\log n}\biggr)>1
	\end{equation*}
for all sufficiently large $n$. This gives the result by Theorem~\ref{thm:sasha1thm1}.
\end{proof}

\begin{theorem} \label{thm:measureone}
Let $u\colon SG_N \to \rr$ and suppose $\Delta_\mu u$ is continuous, where $\mu$ is the standard Bernoulli measure.
Then $\nabla u(\omega)$ is defined $\mu$-a.e.
\end{theorem}
\begin{proof}
We give a crude but sufficient lower bound on the set of words for which $C(\omega,n)$ satisfies the estimate in Theorem~\ref{thm2}.  Suppose we split a word of length $n-(N-2)$ up into disjoint intervals of length $N-1$. Evidently there are at least $k=\frac{n}{N-1}-2$ of these.  The probability of any one such interval being an $N-1$ block is $p_N=(N!)N^{-N}$, so the probability that $C_N(\omega,n)<\frac{\log n}{|\log\beta_N|}:=l$ does not exceed that of having $l$ successes in $k$ binomial trials where success has probability $p_N$.  Using Chernoff's inequality to bound the stated probability by $\exp\bigl(-(kp_N-l)^2/(2kp_N)\bigr)$ and taking $n$ large enough that $kp_N>2l$ we have probability less than $\exp(-kp_N/8)=\exp(-nq_N)$ for a $q_N>0$ that does not depend on $n$.  The latter is summable over $n$, so the bound required in  Theorem~\ref{thm2} follows from the first Borel-Cantelli lemma.
\end{proof}
It is perhaps interesting to note that the preceding reasoning allows one to bound the Hausdorff dimension of the set of points at which $\nabla u$ is undefined by a value strictly less than the Hausdorff dimension of $SG_N$; we omit the details.

\section{Acknowledgements}
The authors are grateful to Alexander Teplyaev and Daniel Kelleher for helpful discussions.

\bibliography{HigherSGgradrefs}
\bibliographystyle{plain}

\end{document}